\documentclass[11pt, oneside]{amsart}   	
\usepackage{geometry}                		
\usepackage{graphicx}				
\usepackage{amssymb}

\usepackage{pst-node}
\usepackage{tikz-cd} 
\usepackage{mathtools}

\usepackage{mathrsfs}

\usepackage{amsthm}

\usepackage{amsfonts}
\usepackage{yhmath}
\usepackage{cleveref}
\usepackage{caption}

\theoremstyle{definition}
\newtheorem{definition}{Definition}[section]

\theoremstyle{remark}
\newtheorem{remark}[definition]{Remark}

\theoremstyle{theorem}
\newtheorem{theorem}[definition]{Theorem}

\theoremstyle{corollary}
\newtheorem{corollary}[definition]{Corollary}

\theoremstyle{lemma}
\newtheorem{lemma}[definition]{Lemma}

\theoremstyle{example}
\newtheorem{example}[definition]{Example}

\theoremstyle{prop}
\newtheorem{prop}[definition]{Proposition}
\author{}

\begin{document}
\title{Some quadratic inequalities on product varieties}

\author{Yucheng Liu}

\address{Beijing International Center for Mathematical Research, Peking University, No.5 Yi-
heyuan Road Haidian District, Beijing, 100871, P.R.China}
\email{liuyucheng@bicmr.pku.edu.cn}


\maketitle 
\begin{abstract}
	Given a (weak) stability condition $\sigma$ on a smooth projective variety X, we can construct a family of stability conditions $\sigma_t$ on $X\times S$ for any smooth projective variety $S$ and any real number $t\geq 0$ as in \cite{Stabilityconditionsonproductvarieties}. In this paper, we prove that any $\sigma_t$ semi-stable object on $X\times S$ satisfies a sequence of quadratic inequalities. If $\sigma$ is the classical slope stability condition, we get sequences of quadratic inequalities for slope semi-stable objects on product varieties. The leading quadratic inequality can be viewed as a weak version of Bogomolov's inequality holding in arbitrary characteristics. 
	
\end{abstract}
	\section{Introduction}

		Let $X$ be a smooth projective surface defined over an algebraically closed field of characteristic $0$, Bogomolov proved his famous inequality in \cite{Bogomolovinequality}, saying that 
		
		$$\Delta(E)=ch_1(E)^2-ch_0(E)ch_2(E)\geq 0$$ for any slope semi-stable sheaf $E$. This can be easily generalized to higher dimensional case by the
		Mehta-Ramanathan restriction theorem. 
		
		In the positive characteristic case, while  Langer \cite{Langer} proved the same inequality holds for strongly $\sigma_H$ semi-stable sheaves (i.e., all the geometric Frobenius pull backs of $E$ are still $\sigma_H$ semi-stable), it is well known that Bogomolov's inequality fails in general (see \cite{Raynaud}).

	In this paper, we will establish sequences of quadratic inequalities of semi-stable objects in $D(X\times S)$, for $X$ and $S$ smooth projective varieties of positive dimension over an algebraically closed field $k$ of  any characteristic.

	Our strategy is following : given a (weak) stability condition $\sigma=(\mathcal{A},Z)$ on $X$, where $\mathcal{A}$ is a Noetherian heart of a bounded t-structure on the derived category $D(X)$ of coherent sheaves on $X$. Abramovich and Polishchuk constructed a global heart $\mathcal{A}_S$ for any complex smooth variety $S$ in \cite{APsheaves}, then Polishchuk refined their construction for any scheme $S$ of finite type over any field $k$ in \cite{polishchuk2007constant}. In \cite{Stabilityconditionsonproductvarieties}, we show that for any object $E\in\mathcal{A}_S$, there is a polynomial $L_E(n)$ naturally associated with $\sigma$. We will call $L_E(n)$ the complexified Hilbert polynomial of $E$ with respect to $\sigma$, we suppress its dependence on $\sigma$ in our notation.
	
 In Section \ref{quadratic inequalitties}, we use the global slicing constructed  in \cite{bayer2014mmp} and the semi-stable reduction sequence in \cite{Stabilityconditionsonproductvarieties} to prove some quadratic inequalities of the coefficients in the complexified Hilbert polynomial. In fact, if we write $L_E(n)$ as:  $$L_E(n)=a_r(E)n^r+a_{r-1}(E)n^{r-1}+\cdots +a_0(E)+i(b_r(E)n^r+b_{r-1}(E)n^{r-1}\cdots+b_0(E))$$ where $r= dim(S)$ and $a_i, b_i$ are linear maps from the Grothendieck group $ K(D(X\times S))$ to $\mathbb{R}$ for any $0\leq i\leq r$. Then it is easy to show that the pair $\sigma_t=(\mathcal{A}_S,a_r-t\cdot b_{r-1}+ib_r)$ is a weak stability condition for any $t\in\mathbb{R}_{\geq 0}$.

 Now, we can state our main theorem in the following way. 
		
			\begin{theorem}\label{Bogomolov's inequality}
				
			If $E\in D(X\times S)$ is semi-stable with respect to $\sigma_t$, then we have the following quadratic inequalities.
				
				(1) $b_r(E)a_{r-1}(E)-b_{r-1}(E)a_r(E)\geq 0$.

				(2) For any $1\leq i\leq r-1$, if $$b_r(E)a_{r-1}(E)-b_{r-1}(E)a_r(E)=\cdots=b_r(E)a_{i}(E)-b_{i}(E)a_r(E)=0,$$ then $b_r(E)a_{i-1}(E)-b_{i-1}(E)a_r(E)\geq 0$.

		\end{theorem}
		
	In section 5, we apply this theorem on the classical slope stability condition. We get a sequence of quadratic inequalities for any slope semi-stable sheaves. Although the leading quadratic inequality is in general weaker than the classical Bogomolov inequality,  it holds in arbitrary characteristics.

		In summary, comparing our results with the  classical ones, there are some advantages and disadvantages. On one side, our approach only deals with product varieties and the leading quadratic inequality is in general weaker than the classical Bogomolov's inequality. On the other side, there are three advantages of our approach: firstly, our approach works in any characteristics; secondly, we give sequences of quadratic inequalities, hence it also provides some constraints of Chern characters in higher degrees; the last one is that our approach can be applied on any (weak) stability conditions, for example, slope stability conditions, tilt stability conditions and any Bridgeland stability conditions.

	\subsection*{Organization of the paper} In Section \ref{WSC}, we review some basic definitions and results in the theory of stability conditions. In Section \ref{Preliminary results}, we review some necessary results in \cite{APsheaves}, \cite{polishchuk2007constant} and \cite{Stabilityconditionsonproductvarieties}.  In Section \ref{quadratic inequalitties}, we use these results to establish a sequence of quadratic inequalities from any weak stability conditions.  Finally, we apply our results to classical slope stability conditions in Section \ref{Applications}.
	\subsection*{Notation and Conventions} In this paper, all varieties are integral separated algebraic schemes of finite type over an algebraically closed field $k$. We will use $D(X)$ rather than the usual notation $D^b(cohX)$ to denote the bounded derived categories of coherent sheaves on $X$. We set $\mathbb{H}=\{z\in\mathbb{C}\mid Im(z)>0\}$. We set $Im(z)$, $Re(z)$, and $Arg(z)$ to be the imaginary part, the real part, and the argument of a complex number $z$ respectively.
	\subsection*{Acknowledgement} I would like to thank Emanuele Macri and Hao Sun for many helpful discussions.
	\section{Weak Stability conditions}\label{WSC}
	The theory of stability conditions introduced by Bridgeland in \cite{bridgeland2007stability}, motivated by Douglas's work on D-branes and $\Pi$-stability \cite{douglas2002dirichlet}. This theory was further studied by Kontsevich and Soibelman in \cite{kontsevich2008stability}.
	In this section, we will review some basic notions in the theory of stability conditions (see \cite{beilinson1982faisceaux}, \cite{bridgeland2007stability}, \cite{kontsevich2008stability} and \cite{bayer2017stability}).
		
		The first notion is $t$-structures on triangulated categories, which was firstly introduced in \cite{beilinson1982faisceaux}. 
	\begin{definition}
		Let $\mathcal{D}$ be an triangulated category. A $t$-structure on $\mathcal{D}$ is a pair of full subcategories $(\mathcal{D}^{\leq 0},\mathcal{D}^{\geq 0})$ satisfying the condition (i), (ii) and (iii) below. We denote $\mathcal{D}^{\leq n}=\mathcal{D}^{\leq 0}[-n]$, $\mathcal{D}^{\geq n}=\mathcal{D}^{\geq 0}[-n]$ for every $n\in\mathbb{Z}$. Then the conditions are:
		
		(i) $Hom(E,F)=0$ for every $E\in\mathcal{D}^{\leq 0}$ and $F\in\mathcal{D}^{\geq 1}$;
		
		(ii) $\mathcal{D}^{\leq -1}\subset \mathcal{D}^{\leq 0}$ and  $\mathcal{D}^{\geq 1}\subset \mathcal{D}^{\geq 0}$.
		
		(iii) every object $E\in\mathcal{D}$ fits into an exact triangle 
		
		$$\tau^{\leq 0}E\rightarrow E\rightarrow \tau^{\geq 1}E\rightarrow \cdots$$ with $\tau^{\leq 0}E\in\mathcal{D}^{\leq 0}$, $\tau^{\geq 1}E\in\mathcal{D}^{\geq 1}$.
		
		The heart of the $t$-structure is $\mathcal{A}=\mathcal{D}^{\leq 0}\cap\mathcal{D}^{\geq 0}$. It is an abelian category (see \cite[Theorem 8.1.9]{hotta2007d}). The associated cohomology functors are defined by $H^0(E)=\tau^{\leq 0}\tau^{\geq 0}E$, $H^i(E)=H^0(E[i])$. We will also need the notation $\mathcal{D}^{[a,b]}=\mathcal{D}^{\leq b}\cap\mathcal{D}^{\geq a}$.
	\end{definition}
	
Combining this definition with Harder-Narasimhan filtrations, Bridgeland defined the notion of stability conditions on a triangulated category in \cite{bridgeland2007stability}.
	
	\begin{definition}\label{slicing}
		A stability condition $(\mathcal{P},Z)$  on a triangulated category $\mathcal{D}$ consists of a group homomorphism $Z:K(\mathcal{D})\rightarrow\mathbb{C}$ called the central charge, and full subcategories $\mathcal{P}(\phi)\in\mathcal{D}$ for each $\phi\in\mathbb{R}$, satisfying the following axioms:
		
		\par
		
		(a) if $E\in\mathcal{P}(\phi)$ is a nonzero object, then $Z(E)=m(E)exp(i\pi\phi)$ for some $m(E)\in\mathbb{R}_{>0}$,
		\par
		(b) for all $\phi \in \mathbb{R}$, $\mathcal{P}(\phi+1)=\mathcal{P}(\phi)[1]$,
		
		\par
		(c) if $\phi_1>\phi_2$ and $A_j\in\mathcal{P}(\phi_j)$ then $Hom_{\mathcal{D}}(A_1,A_2)=0$,
		
		\par

		(d) for every $0\neq E\in\mathcal{D}$ there exist a finite sequence of real numbers
		
		$$\phi_1>\phi_2>\cdots>\phi_m$$and a sequence of morphisms 
		
		$$0=E_0\xrightarrow{f_1}E_1\xrightarrow{f_2} \cdots \xrightarrow{f_m}E_m=E $$such that the cone of $f_j$ is in $\mathcal{P}(\phi_j)$ for all $j$.
	\end{definition}

	\begin{remark}

		If we allow $m(E)$ to be $0$ for $\phi\in\mathbb{Z}$ in (a), then the pair $(\mathcal{P},Z)$ is called a weak stability condition. In \cite{kontsevich2008stability}, the authors require the pair $(\mathcal{P},Z)$ to satisfy one extra condition (support property) to be a stability condition. We do not include this condition because it is not needed in this paper.
	\end{remark}
	   The data $\mathcal{P}$ of full subcategories $\mathcal{P}(\phi)$ is called a slicing on $\mathcal{D}$, a slicing can be viewed as a refined  $t$-structure on a triangulated category. Indeed, one can easily check that a slicing on $\mathcal{D}$ gives us a lot of $t$-structures on $\mathcal{D}$: for any $\phi\in\mathbb{R}$, we have a $t$-structure $(\mathcal{P}(>\phi-1),\mathcal{P}(\leq \phi))$ on $\mathcal{D}$.

In particular, a slicing $\mathcal{P}$ of $\mathcal{D}$ provides us a heart $\mathcal{P}(0,1]=\mathcal{P}(>0)\cap \mathcal{P}(\leq 1)$. Hence, a weak stability condition $(\mathcal{P},Z)$ gives us a pair $(\mathcal{A},Z)$, where $\mathcal{A}$ is an abelian category. This construction results in an equivalent definition of stability conditions.

\begin{definition}\label{second definition}
	A  stability condition on $\mathcal{D}$ is a pair $\sigma=(\mathcal{A},Z)$ consisting of the heart of a bounded t-structure $\mathcal{A}\subset\mathcal{D}$ and a group homomorphism $Z:K(A)\rightarrow \mathbb{C}$ such that  (a) and (b) below are satisfied:
	\par
	(a) (Positivity condition) For any nonzero object $E\in \mathcal{A}$, we have $Im(Z(E))\geq0$, with the property $Im(Z(E))=0 \implies  Re(Z(E))<0$. 
	
	\par
	
(b) (HN property) The function $Z$ allow us to define a slope for any object $E$ in the heart $\mathcal{A}$ by
	
	$$\mu_{\sigma}(E):=\begin{cases} -\frac{Re(Z(E))}{Im(Z(E))}\ &\text{if} \  Im(Z(E))> 0,\\ +\infty &\text{otherwise.} \end{cases}$$

	The slope function gives a notion of stability: A nonzero object $E\in \mathcal{A}$ is $\sigma$ semi-stable if for every proper subobject $F$, we have $\mu_{\sigma}(F)\leq\mu_{\sigma}(E)$.
	
	We require any object $E$ of $\mathcal{A}$ to have a Harder-Narasimhan filtration in $\sigma$ semi-stable ones, i.e., there exists a unique filtration $$0=E_0\subset E_1 \subset E_2\subset \cdots \subset E_{m-1} \subset E_m=E$$ such that $E_i/E_{i-1}$ is  $\sigma$ semi-stable and $\mu_{\sigma}(E_i/E_{i-1})>\mu_{\sigma}(E_{i+1}/E_i)$ for any $1\leq i\leq m$.

\end{definition}

\begin{remark}
	Similarly, we call $(\mathcal{A},Z)$ a weak stability condition if we allow $Z(E)=0$ for nonzero object $E\in\mathcal{A}$. 
	
	If the image of imaginary part of $Z$  is discrete in $\mathbb{R}$, and $\mathcal{A}$ is Noetherian, then the HN property is satisfied automatically. See \cite[Proposition 2.4]{bridgeland2007stability}, \cite[Lemma 4.9]{macri2017lectures}).
\end{remark}

	In this paper, we are  interested in the case when $\mathcal{D}$ is the bounded derived category of coherent sheaves on an algebraic variety $X$. From now on, $X$ will be a smooth projective variety over an algebraically closed field $k$, and $D(X)$ will be the bounded derived category of coherent sheaves on $X$. In this case, the construction of stability conditions is still a big open problem in this theory (for some special projective varieties, see \cite{Stabilityconditiononcurves}, \cite{bridgeland2008stability},  \cite{ABsurfaces}, \cite{baye2011bridgeland}, \cite{AbelianI}, \cite{AbelianII},  \cite{bayer2016space},  \cite{Macrifano}, \cite{Koseki},  \cite{li2018stability}, \cite{Lifano3fold} and \cite{Stabilityconditionsonproductvarieties}). However, we can always find weak stability conditions on $D(X)$.
	
	\begin{example}\label{slope stability conditions}
		Suppose $X$ is a smooth projective variety of dimension $d$, $H$ is an ample divisor on $X$. The  pair $\sigma_H=(cohX, -H^{d-1}ch_1(E)+iH^dch_0(E))$ is a weak stability condition on $D(X)$. Classically, the $\sigma_H$ semi-stable objects are called slope semi-stable with respect to $H$. Note that this pair is not a stability condition unless $X$ is a curve.
	\end{example}

	\section{Global heart and complexified Hilbert polynomial}\label{Preliminary results}
In this section, we will recall some constructions and results from \cite{APsheaves}, \cite{polishchuk2007constant}, \cite{bayer2014mmp} and \cite{Stabilityconditionsonproductvarieties}. We will work under the following setup in the rest of this paper. 
	
\textbf{Setup}: Suppose $X$ and $S$ are smooth projective varieties, and $\sigma=(\mathcal{A},Z)$ is a weak stability condition on $D(X)$, where $\mathcal{A}$ is Noetherian and the image of $Z$ is discrete. The global heart $\mathcal{A}$ corresponds to a $t$-structure $(D^{\leq 0}(X), D^{\geq 0}(X))$ on $D(X)$.

For any t-structure  $(D^{\leq 0}(X), D^{\geq 0}(X))$ on $D(X)$, we have the following theorem. 

\begin{theorem}[{\cite[Theorem 3.3.6]{polishchuk2007constant}}]
	
	 Suppose $S$ is a projective variety of dimension $r$, and $\mathcal{O}(1)$ is an ample line bundle on $S$. There exists a global t-structure on $D(X\times S)$ defined as 
		$$D^{[a,b]}(X\times S)=\{E\in D(X\times S)\mid \textbf{R}p_*(E \otimes q^*(\mathcal{O}(n)))\in D^{[a,b]}(X) \ for\ all \ n\gg 0\}.$$
		
	Here $a,b$ can be infinite. Moreover, the global heart $$\mathcal{A}_S=D^{\leq 0}(X\times S)\cap D^{\geq 0}(X\times S)$$ is Noetherian and independent of the choice of $\mathcal{O}(1)$.

\end{theorem} 
\begin{remark}
	This theorem can be viewed as a generalization of Serre's vanishing. In fact, if $\mathcal{A}=coh(X)$ is the abelian category of coherent sheaves on $X$, then $\mathcal{A}_S=coh(X\times S)$ is the abelian category of coherent sheaves on $X\times S$.
\end{remark}
In \cite{Stabilityconditionsonproductvarieties}, we observed that $Z(\textbf{R}p_*(E\otimes \mathcal{O}(n)))$ is a polynomial of degree no more than $dim(S)=r$, whose leading coefficient is a weak stability function on $\mathcal{A}_S$. We denote this polynomial by $L_E(n)$ for any object $E$ in $\mathcal{A}_S$.  These can be stated as following.

\begin{theorem}[{\cite[Theorem 3.3]{Stabilityconditionsonproductvarieties}}]\label{glabal weak stablity condition}
	Assume $S$ is a smooth projective variety of dimension $r$, we define $(\mathcal{A}_S, Z_S)$ as below:
	\\
	$$\mathcal{A}_S=\{E\in D(X\times S)\mid \textbf{R}p_*(E \otimes q^*(\mathcal{O}(n)))\in \mathcal{A}\  for\ all\ n\gg0 \}$$
	$$Z_S(E)=\lim_{n\rightarrow +\infty}\frac{Z(\textbf{R}p_*(E\otimes q^*(\mathcal{O}(n)))r!}{n^r vol{(\mathcal{O}(1))}},$$where $vol(\mathcal{O}(1))$ is the volume of $\mathcal{O}(1)$. Then this pair is a weak stability condition on $D(X\times S)$.
\end{theorem}

\begin{remark}
		If we take $X=Spec(\mathbb{C})$, $\mathcal{A}$ is the category of $\mathbb{C}$-vector spaces and $Z(V)=z\cdot dim(V)$ for any finite dimensional $\mathbb{C}$ vector space, where $z\in\mathbb{H}\cup\mathbb{R}_{\leq 0}$. Then the global heart is the category of coherent sheaves on $S$, and $L_E(n)=z\cdot Hilb_E(n)\ for\ n\gg 0$. Therefore, we call $L_E(n)$ the complexified Hilbert polynomial.
\end{remark}

We get the following slope function $\mu_1$ from the pair $(\mathcal{A}_S,Z_S)$:

$$\mu_1(E)\coloneqq\begin{cases} -\frac{Re(Z_S(E))}{Im(Z_S(E))} &\text{if} \ Im(Z_S(E))> 0,\\ +\infty &\text{otherwise.}\end{cases}$$

This weak stability condition is closely related to the global slicing constructed in \cite[Section 4]{bayer2014mmp}. The global slicing will play an important role in this paper, so we include the explicit construction of the global slicing in the next. 

Given a stability condition $\sigma=(Z, \mathcal{P})$ on $D(X)$ and a phase $\phi\in\mathbb{R}$, we have its associated t-structure $$(\mathcal{P}(>\phi)=\mathcal{D}_{\phi}^{\leq -1}(X) ,\mathcal{P}(\leq\phi)=\mathcal{D}_{\phi}^{\geq 0}(X))$$ on $D(X)$. By Abramovich and Polishchuk's construction, we get a global t-structure $$(\mathcal{P}_S(>\phi)\coloneqq\mathcal{D}_{\phi}^{\leq -1}(X\times S), \mathcal{P}_S(\leq \phi))\coloneqq\mathcal{D}_{\phi}^{\geq 0}(X\times S))$$ on $D(X\times S)$. Then we have the following lemma in \cite[Section 4]{bayer2014mmp}.

\begin{lemma}[{\cite[Lemma 4.6]{bayer2014mmp}}]\label{global slicing}
	Assume $\sigma=(Z, \mathcal{P})$ is a weak stability condition as in our setup, and $\mathcal{P}_S(>\phi)$, $\mathcal{P}_S(\leq \phi)$ defined as above. There is a slicing $\mathcal{P}_S$ on $D^b(X\times S)$ defined by
	
	$$\mathcal{P}_S(\phi)=\mathcal{P}_S(\leq \phi)\cap \underset{\epsilon>0}{\cap}\mathcal{P}_S(>\phi-\epsilon).$$
\end{lemma}

To conclude this section, we describe the relation between $\mu_1$ semi-stable objects and the global slicing $\mathcal{P}_S$. 
	
	\begin{prop}[{\cite[Proposition 3.14]{Stabilityconditionsonproductvarieties}}]\label{semistable reduction}
		If  $S$ is a smooth projective variety, and $E\in \mathcal{A}_S$ is semi-stable with respect to $\mu_1$ of phase $\phi$ and $Z_S(E)\neq 0$, then there exists a short exact sequence 
		
		$$0\rightarrow K\rightarrow E\rightarrow Q\rightarrow 0$$
		such that $K\in \mathcal{P}_S(\phi)$, $Q\in \mathcal{P}_S(<\phi)$ and $Z_S(Q)=0$.
	\end{prop}

	\section{Positivity of the coefficients}\label{quadratic inequalitties}
	
	In this section, we will investigate some positivity of the coefficients in the complexified Hilbert polynomial $L_E(n)$. 
	
	Under the same assumption in Section \ref{Preliminary results}, we can write the polynomial $L_E(n)$ as:  $$L_E(n)=a_r(E)n^r+a_{r-1}(E)n^{r-1}+\cdots +a_0(E)+i(b_r(E)n^r+b_{r-1}(E)n^{r-1}\cdots+b_0(E))$$ where $a_i, b_i$ are linear maps from $ K(D(X\times S))$ to $\mathbb{R}$ for any $0\leq i\leq r$. We have the following positivity condition.
	
	\begin{lemma}\label{linear inequality}
		Suppose $E$ is an object in $\mathcal{A}_S$, then 
		
			(1) $b_r(E)\geq 0$.
			
		(2) If $b_r(E)=0$, then $a_r(E)\leq0$ and $b_{r-1}(E)\geq 0$.

			(3) In general, if $$b_r(E)=a_r(E)=b_{r-1}(E)=\cdots=a_i(E)=b_{i-1}(E)=0,$$ then $a_{i-1}(E)\leq 0$ and $b_{i-2}(E)\geq 0$ for any $2\leq i\leq r$.
			
			(4) Moreover, if $\sigma=(\mathcal{A},Z)$ is a stability condition on $D(X)$, and $E$ is a nonzero object in $\mathcal{A}_S$, then $$b_r(E)=a_r(E)=b_{r-1}(E)=\cdots=a_1(E)=b_{0}(E)=0$$ implies $a_0(E)<0$.
	\end{lemma}
	\begin{proof}
			Since $E\in\mathcal{A}_S=\mathcal{P}_S(0,1]$, we have the HN-filtration of $E$ with respect to the slicing $\mathcal{P}_S$. We get $$0=E_0\subset E_1 \subset E_2\subset \cdots \subset E_{m-1} \subset E_m=E,$$ where $E_{i}/E_{i-1}\in\mathcal{P}_S(\phi_i)$ and $\phi_i\in(0,1]$. 
			
			By Lemma \ref{global slicing}, the argument $Arg(L_{E_i/E_{i-1}}(n))$ lies in $(\pi(\phi_i-\epsilon),\pi\phi_i]$ for any $\epsilon>0$ and $n\gg 0$ . This implies that the argument of the first nonzero coefficient of $L_{E_i/E_{i-1}}(n)$ is $\pi\phi_i$. Therefore, the argument of the first  nonzero coefficient of $L_{E_i}(n)$ is in $(0,\pi]$. Combing this with the fact $L_E(n)\in\mathbb{H}\cup\mathbb{R}_{<0}$ for $n\gg 0$, we get (1), (2) and (3). 
			
			For (4), it is because $L_{E}(n)=0$ implies $p_*(E\otimes q^*(\mathcal{O}(n)))=0$ for $n\gg0$, which implies $E=0$.
		
	\end{proof}
	
	The definition of $\mathcal{P}_S(\phi)$ also gives us the following sequence of quadratic inequalities.
	\begin{lemma}\label{first quadratic inequlity}
		Suppose $E$ is an object in $\mathcal{P}_S(\phi)$, then we have the following inequalities;
		
		(1) $b_r(E)a_{r-1}(E)-b_{r-1}(E)a_r(E)\geq 0$.

		(2) In general, for any $1\leq i\leq r-1$, if $$b_r(E)a_{r-1}(E)-b_{r-1}(E)a_r(E)=\cdots=b_r(E)a_{i}(E)-b_{i}(E)a_r(E)=0,$$ then $b_r(E)a_{i-1}(E)-b_{i-1}(E)a_r(E)\geq 0$.
		
	\end{lemma}
	
	\begin{proof} If $a_r(E)+ib_r(E)=0$, the lemma is vacuous. Hence we can assume that $a_r(E)+ib_r(E)\neq 0$. This implies that $Arg(a_r(E)+ib_r(E))$ is $\pi\phi$ by the proof of last lemma. 
	
		Recall the definition $$\mathcal{P}_S(\phi)=\mathcal{P}_S(\leq \phi)\cap \underset{\epsilon>0}{\cap}\mathcal{P}_S(>\phi-\epsilon).$$ 
		Hence, we know that for any small enough $\epsilon>0$, there exists a positive integer $N_0$, such that for any positive integer $n>N_0$,  $Arg(L_E(n))$ is in $(\pi(\phi-\epsilon),\pi\phi]$.
			
			Therefore, the complex number $a_{r-1}(E)+ib_{r-1}(E)$ is not on the left side of the line passing from $0$ to $a_r(E)+ib_r(E)$ in the complex plane. Otherwise, we could find a positive integer $N_1$, such that for any integer $n>N_1$, the $Arg(L_E(n))$ is strictly bigger that $\pi\phi$, this contradicts to the fact $E\in\mathcal{P}_S(\phi)$. This implies $$b_r(E)a_{r-1}(E)-a_r(E)b_{r-1}(E)\geq 0.$$
			
			If $b_r(E)a_{r-1}(E)-a_r(E)b_{r-1}(E)= 0$, this means that $a_{r-1}(E)+ib_{r-1}(E)$ is on the line from $0$ to $ (a_r(E)+ib_r(E))$.   By the same argument, we show that  $a_{r-2}(E)+ib_{r-2}(E)$ is not on the left side of the line passing from $0$ to $a_r(E)+ib_r(E)$. Hence,  the following inequality holds.
			
			$$b_r(E)a_{r-2}(E)-a_r(E)b_{r-2}(E)\geq 0.$$
			
			Inductively using the same argument, we proved (2).
	\end{proof}
	
	\begin{remark}
In the case when $X$ is point, these inequalities are vacuous. Indeed, all complex numbers $a_j(E)+ib_j(E)$ are on the same line in the complex plane, for any $0\leq j\leq r$. This first inequality in (1) also follows from the positivity lemma in \cite{bayer2014projectivity}. 
	\end{remark}
	Combining last Lemma with Lemma \ref{semistable reduction}, we get the following theorem.
	
	\begin{lemma}\label{second quadratic inequality}
			If $E\in\mathcal{A}_S$ is semi-stable with respect to $\mu_1$, then the following inequalities are satisfied.

			(1) $b_r(E)a_{r-1}(E)-b_{r-1}(E)a_r(E)\geq 0$.

			(2) In general, if $$b_r(E)a_{r-1}(E)-b_{r-1}(E)a_r(E)=\cdots=b_r(E)a_{i}(E)-b_{i}(E)a_r(E)=0,$$ Then, $b_r(E)a_{i-1}(E)-b_{i-1}(E)a_r(E)\geq 0$ for any $1\leq i\leq r-1 $.
	\end{lemma}

	\begin{proof}
				If $b_r(E)=0$, then (1) follows from Lemma \ref{linear inequality}. For (2), let us assume that $$b_r(E)a_{r-1}(E)-b_{r-1}(E)a_r(E)=\cdots=b_r(E)a_{i}(E)-b_{i}(E)a_r(E)=0,$$ then either $a_r(E)=0$ or $$b_{r-1}(E)=\cdots=b_{i}(E)=0.$$ In the first case, (2) becomes vacuous. The second case implies $b_{i-1}(E)\geq 0$ since $$L_E(n)\in\mathbb{H}\cup\mathbb{R}_{\leq 0}$$ for $n\gg 0$, so we also have $b_r(E)a_{i-1}(E)-b_{i-1}(E)a_r(E)\geq 0$ by Lemma \ref{linear inequality}.
			
			Therefore, we can assume $b_r(E)> 0$. By Proposition \ref{semistable reduction}, there exists a short exact sequence	in $\mathcal{A}_S$ $$0\rightarrow K\rightarrow E\rightarrow Q\rightarrow 0$$
			such that $K\in \mathcal{P}_S(\phi)$, $Q\in \mathcal{P}_S(<\phi)$ and $Z_S(Q)=0$.

			If $Q=0$, the statement follows from Lemma \ref{first quadratic inequlity}.

			Therefore. we can assume $Q$ is a nonzero object in $\mathcal{P}_S(<\phi)$. By the fact  $K\in \mathcal{P}_S(\phi)$, $Q\in \mathcal{P}_S(<\phi)$, we get  that if $a_{r-1}(Q)+ib_{r-1}(Q)\neq 0$, then $$\frac{-a_r(E)}{b_r(E)}>\frac{-a_{r-1}(Q)}{b_{r-1}(Q)},$$
which is equivalent to $$b_r(E)a_{r-1}(Q)-a_r(E)b_{r-1}(Q)> 0.$$			
			Lemma \ref{first quadratic inequlity} and Proposition \ref{semistable reduction}  imply the following inequality.
			
			$$b_r(E)a_{r-1}(K)-b_{r-1}(K)a_r(E)=b_r(K)a_{r-1}(K)-b_{r-1}(K)a_r(K)> 0.$$
			
		Adding these two inequalities, we get (1).

		Now suppose 
			
			$$b_r(E)a_{r-1}(E)-b_{r-1}(E)a_r(E)= 0,$$ This implies $$b_r(K)a_{r-1}(K)-b_{r-1}(K)a_r(K)=0$$ and $a_{r-1}(Q)=b_{r-1}(Q)=0$.
			
			Then by Lemma \ref{first quadratic inequlity}, we get $$b_r(K)a_{r-2}(K)-a_r(K)b_{r-2}(K)\geq 0.$$  The fact $Q\in\mathcal{P}_S(<\phi)$ implies that if $a_{r-2}(Q)+ib_{r-2}(Q)\neq 0$, then $$\frac{-a_r(E)}{b_r(E)}>\frac{-a_{r-2}(Q)}{b_{r-2}(Q)}.$$
			
		Similarly, we have $$b_r(E)a_{r-2}(E)-b_{r-2}(E)a_r(E)\geq 0$$.
			
			Inductively using the same argument, we get (2).
			
	\end{proof}
	
	By Lemma \ref{linear inequality}, the $ a_r-t\cdot b_{r-1}+ib_r$ is a weak stability function on $\mathcal{A}_S$ for any $t\in\mathbb{R}_{\geq 0}$. Moreover, since $\mathcal{A_S}$ is Noetherian and $b_r$ has discrete image, the HN property is satisfied. Hence, the pair $\sigma_t=(\mathcal{A}_S,a_r-t\cdot b_{r-1}+ib_r)$ is a weak stability condition for any $t\in\mathbb{R}_{\geq 0}$.

	\begin{theorem}\label{third quadratic inequality}
		For any $t\in\mathbb{R}_{\geq 0}$, if $E\in\mathcal{A}_S$ is semi-stable with respect to $\sigma_t$, the following inequalities are satisfied.

		(1) $b_r(E)a_{r-1}(E)-b_{r-1}(E)a_r(E)\geq 0$.

		(2) In general, if $$b_r(E)a_{r-1}(E)-b_{r-1}(E)a_r(E)=\cdots=b_r(E)a_{i}(E)-b_{i}(E)a_r(E)=0,$$ Then, $b_r(E)a_{i-1}(E)-b_{i-1}(E)a_r(E)\geq 0$ for any $1\leq i\leq r-1 $.
	\end{theorem}
	
	\begin{proof}

	As in the proof of Lemma \ref{second quadratic inequality}, we can assume that $b_r(E)>0$. Moreover, if $t=0$, this is  Lemma \ref{second quadratic inequality}. Hence we can also assume $t>0$. 
	
	The following proof is essentially the same as the proof in \cite[Lemma 5.5]{Stabilityconditionsonproductvarieties}, we include the details for reader's convenience.  Take the HN filtration of $E$ with respect to $\mu_1$, we get the sequence  $$0=E_0\subset E_1 \subset E_2\subset \cdots \subset E_{l-1} \subset E_l=E,$$ and use $Q_k$ to denote $E_{k}/E_{k-1}$ for $1\leq k\leq l$. 
	
	 Since $E$ is $\sigma_t$ semi-stable and $b_r(E)$ is positive, we get  $b_r(Q_1)>0$. Furthermore, by the definition of HN filtration, we have
	 
	 \begin{equation}
	 \frac{-a_r(Q_1)}{b_r(Q_1)}>\frac{-a_r(Q_2)}{b_r(Q_2)}>\cdots>\frac{-a_r(Q_l)}{b_r(Q_l)}.
	 \end{equation} 
	 
	Hence, we get $b_r(Q_k)>0$ for any $1\leq k\leq l$.  Also, by Lemma \ref{second quadratic inequality}, we have \begin{equation}
	  b_r(Q_k)a_{r-1}(Q_k)-b_{r-1}(Q_k)a_r(Q_k)\geq 0
	  \end{equation} for any $1\leq k\leq l$. 
	  
	  The last piece of data is that $E$ is $\sigma_t$ semi-stable. Hence \begin{equation}
	  \frac{\sum\limits_{k=1}^j(-a_r(Q_k)+tb_{r-1}(Q_k))}{\sum\limits_{k=1}^jb_r(Q_k)}\leq \frac{-a_r(E)+tb_{r-1}(E)}{b_r(E)}\leq \frac{\sum\limits_{k=j}^l(-a_r(Q_k)+tb_{r-1}(Q_k))}{\sum\limits_{k=j}^lb_r(Q_k)}
	  \end{equation} for any $1\leq j\leq l$.
	 
	 We can prove that \begin{equation*}
	\begin{split}
&	a_{r-1}(E)\geq \sum\limits_{k=1}^l \frac{a_r(Q_k)b_{r-1}(Q_k)}{b_r(Q_k)} \\ & =\frac{a_r(Q_l)}{b_r(Q_l)}b_{r-1}(E)-\sum\limits_{j=1}^{l-1}\sum\limits_{k=1}^jb_{r-1}(Q_k)(\frac{a_r(Q_{j+1})}{b_r(Q_{j+1})}-\frac{a_r(Q_j)}{b_r(Q_j)} )\\ & \geq \frac{a_r(Q_l)}{b_r(Q_l)}b_{r-1}(E)-\frac{1}{t}\sum\limits_{j=1}^{l-1}\sum\limits_{k=1}^j(\frac{-a_r(E)+tb_{r-1}(E)}{b_r(E)}b_r(Q_k)+a_r(Q_k))(\frac{a_r(Q_{j+1})}{b_r(Q_{j+1})}-\frac{a_r(Q_j)}{b_r(Q_j)} )\\ &  =\frac{1}{t}\sum\limits_{k=1}^l\frac{a_r(Q_k)}{b_r(Q_k)}(\frac{-a_r(E)+tb_{r-1}(E)}{b_r(E)}b_r(Q_k)+a_r(Q_k)) \\& =\frac{1}{t}(a_r(E)\frac{-a_r(E)+b_{r-1}(E)t}{b_r(E)}+\sum\limits_{k=1}^l\frac{a_r(Q_k)^2}{b_r(Q_k)}).
	\end{split}
	 \end{equation*}
	 
	 	The first inequality is from (2) and the fact $b_r(Q_k)>0$, the first equality is Abel's summation formula. The second inequality comes from (1) and the left side of (3). The second equality is Abel's summation formula. Hence, we have the following inequality
	 
	 	\begin{equation*}
		\begin{split}
		b_r(E)a_{r-1}(E)-a_r(E)b_{r-1}(E) &\geq \frac{1}{t}(\sum\limits_{k=1}^lb_r(Q_k)\sum\limits_{k=1}^l\frac{a_r(Q_k)^2}{b_r(Q_k)}-a_r(E)^2)\\ & =\frac{1}{t}\sum\limits_{1\leq i<j\leq l}(\frac{a_r(Q_i)}{\sqrt{b_r(Q_i)}}\sqrt{b_r(Q_j)}-\frac{a_r(Q_j)}{\sqrt{b_r(Q_j)}}\sqrt{b_r(Q_i)})^2.
		\end{split}	
		\end{equation*}
		
		Hence $b_r(E)a_{r-1}(E)-a_r(E)b_{r-1}(E)\geq 0$, and the equality holds only if $$\frac{a_r(Q_i)}{\sqrt{b_r(Q_i)}}\sqrt{b_r(Q_j)}=\frac{a_r(Q_j)}{\sqrt{b_r(Q_j)}}\sqrt{b_r(Q_i)}$$  for any $1\leq i,j\leq l$, which is equivalent to $$\frac{a_r(Q_i)}{b_r(Q_i)}=\frac{a_r(Q_j)}{b_r(Q_j)}$$ any $1\leq i,j\leq l$. This contradicts the definition of HN-filtration unless $l=1$, or equivalently, $E$ is semi-stable with respect to $\mu_1$. Therefore, Lemma \ref{second quadratic inequality} implies (2).
	\end{proof}
	\section{Quadratic inequalities}\label{Applications}

	Let $X$ be a smooth projective variety of positive dimension $d$, and let $H$ be an ample divisor on $X$. Then the  pair $\sigma_H=(coh(X), -H^{d-1}ch_1(E)+iH^dch_0(E))$ is a weak stability condition on $D(X)$.

		Applying the construction and results in Section \ref{Preliminary results} and Section \ref{quadratic inequalitties}, we get two sequences of quadratic inequalities.
		
		\begin{theorem}\label{Bogomolov's inequality}
			
			Let $X, S$ be smooth projective varieties of positive dimension $d, r$ respectively, and $p,q$ be the projections from $X\times S$ to $X$ and $S$ respectively. Suppose we have $$H_1=c_1(p^*\mathcal{O}_X(1)), H_2=c_1(q^*\mathcal{O}_S(1)),$$ where $\mathcal{O}_X(1),\mathcal{O}_S(1)$ are ample line bundles on $X$ and $S$ respectively. 
			Then if $E$ is semi-stable sheaf with respect to $m_1H_1+m_2H_2$, where $m_1,m_2\in\mathbb{Z}_{>0}$ are two positive integers, we have the following inequalities.
			
			(1) $b_r(E)a_{r-1}(E)-b_{r-1}(E)a_r(E)\geq 0$.

			(2) For any $1\leq i\leq r-1$, if $$b_r(E)a_{r-1}(E)-b_{r-1}(E)a_r(E)=\cdots=b_r(E)a_{i}(E)-b_{i}(E)a_r(E)=0,$$ then $b_r(E)a_{i-1}(E)-b_{i-1}(E)a_r(E)\geq 0$.

			Here $$a_k(E)=-\frac{1}{k!}H_1^{d-1}H_2^k\cdot \sum\limits_{\substack{i+j=r+1-k \\ i,j\geq 0}} td_ich_j(E),$$ $$b_k(E)=\frac{1}{k!}H_1^{d}H_2^k\cdot \sum\limits_{\substack{i+j=r-k\\ i,j\geq 0}} td_ich_j(E)$$  for any $0\leq k\leq r$, and $td_i$ denotes the i-th Todd class of the relative tangent bundle of $p$.
			
			(3)  $b'_d(E)a'_{d-1}(E)-b'_{d-1}(E)a'_d(E)\geq 0$.
			
				(4)For any $1\leq i\leq d-1$, if $$b'_d(E)a'_{d-1}(E)-b'_{d-1}(E)a'_d(E)=\cdots=b'_d(E)a'_{i}(E)-b'_{i}(E)a'_d(E)=0,$$ then $b'_d(E)a'_{i-1}(E)-b'_{i-1}(E)a'_d(E)\geq 0$.

				Here $$a_k'(E)=-\frac{1}{k!}H_1^{k}H_2^{r-1}\cdot \sum\limits_{\substack{i+j=d+1-k \\ i,j\geq 0}} td'_ich_j(E),$$ $$b_k'(E)=\frac{1}{k!}H_1^{k}H_2^r\cdot \sum\limits_{\substack{i+j=d-k\\ i,j\geq 0}} td'_ich_j(E)$$ for any $0\leq k\leq d$, and $td'_i$ denotes the i-th Todd class of the relative tangent bundle of $q$.
		\end{theorem}
		
		\begin{proof}By the symmetry of $X$ and $S$, we only need to prove (1) and (2). 
		Firstly, the global heart $\mathcal{A}_S$ is $coh(X\times S)$ since $\mathcal{A}=coh(X)$. 
		
		We need to calculate the coefficients of the polynomial $L_E(n)$ for $E\in\mathcal{A}_S$. To calculate the polynomial $L_E(n)$, we need to calculate the Chern characters of $\textbf{R}p_*(E\times q^*(\mathcal{O}(n)))$. The calculation can be done by using Gronthendieck-Riemann-Roch formula:
		
		\begin{equation*}
			\begin{split}
		&	ch(\textbf{R}p_*(E\otimes q^*\mathcal{O}_S(n)))  =p_*(ch(E)ch(q^*\mathcal{O}_S(n))td(T_p)) \\ & =p_*((ch_0(E),ch_1(E),\cdots,ch_{d+r}(E))(1,nH_2, \cdots, \frac{n^r}{r!}H_2^r,0\cdots,0)(1,td_1,\cdots,td_r,0,\cdots, 0))
			\\&=p_*((ch_0(E),\cdots,\ \sum\limits_{\substack{i+j=r \\i, j \geq 0}}ch_i(E)\frac{n^j}{j!}H_2^j, \cdots, \sum\limits_{\substack{i+j=d+r \\i,j\geq 0}}ch_i(E)\frac{n^j}{j!}H_2^j)(1,td_1,\cdots,td_r,0\cdots,0))
		\\&=(p_*(\sum\limits_{\substack{i+j+k=r \\i, j,k\geq 0}}ch_i(E)\frac{n^j}{j!}H_2^jtd_k),p_*(\sum\limits_{\substack{i+j+k=r+1 \\i, j,k \geq 0}}ch_i(E)\frac{n^j}{j!}H_2^jtd_k),\cdots)
			\end{split}
				\end{equation*}
				
				Apply the weak stability function $-H^{d-1}ch_1+i\cdot H^d ch_0$, we get 
				
				$$L_E(n)=-H_1^{d-1}\sum\limits_{\substack{i+j+k=r+1 \\i, j,k \geq 0}}ch_i(E)\frac{n^j}{j!}H_2^jtd_k+i\cdot H_1^d \sum\limits_{\substack{i+j+k=r \\i, j,k\geq 0}}ch_i(E)\frac{n^j}{j!}H_2^jtd_k$$ by projection formula.
				
				Hence, the coefficients $a_k, b_k$ in the polynomial $L_E(n)$ can be written as 
				 $$a_k(E)=-\frac{1}{k!}H_1^{d-1}H_2^k\cdot \sum\limits_{\substack{i+j=r+1-k \\ i,j\geq 0}} td_ich_j(E),$$ $$b_k(E)=\frac{1}{k!}H_1^{d}H_2^k\cdot \sum\limits_{\substack{i+j=r-k\\ i,j\geq 0}} td_ich_j(E)$$ for any $1\leq k\leq r$.
				
				The last step is to show that there exists $t\in\mathbb{R}_{\geq 0}$, such that the weak stability condition $\sigma_t=(\mathcal{A}_S, a_r-tb_{r-1}+ib_r)$ is equivalent to the slope stability with respect to $m_1H_1+m_2H_2$.  
				Since we have $$a_r(E)=-\frac{1}{r!}H_1^{d-1}H_2^r\cdot(td_1ch_0(E)+ch_1(E))=-\frac{1}{r!}H_1^{d-1}H_2^r\cdot ch_1(E),$$
				$$b_r(E)=\frac{1}{r!}H_1^{d}H_2^r\cdot ch_0(E),$$
				
				$$b_{r-1}=\frac{1}{(r-1)!}H_1^{d}H_2^{r-1}\cdot(td_1ch_0(E)+ch_1(E)).$$ The slope function is \begin{equation*}
				\begin{split}
				\mu_{\sigma_t}(E)&=\frac{-a_r(E)+tb_{r-1}(E)}{b_r(E)} \\ &=\frac{H_1^{d-1}H_2^r\cdot ch_1(E)+t \cdot rH_1^dH_2^{r-1}(ch_0(E)td_1+ch_1(E))}{H_1^dH_2^r\cdot ch_0(E)} \\ &=\frac{(H_1^{d-1}H_2^r+t\cdot rH_1^dH_2^{r-1})ch_1(E)}{H_1^dH_2^r\cdot ch_0(E)}+\frac{t\cdot rH_1^dH_2^{r-1}td_1}{H_1^dH_2^r}.
				\end{split}
				\end{equation*}
				The last term in the last line is independent of $E$, so it will not affect the stability of $E$. If we take $t=\frac{m_1}{m_2d}$, we can easily check 
				\begin{equation*}
				\begin{split}
				H_1^{d-1}H_2^r & +t\cdot rH_1^dH_2^{r-1}=H_1^{d-1}H_2^r+\frac{m_1r}{m_2d}H_1^dH_2^{r-1}\\&=\frac{1}{\binom{d+r-1}{r}m_1^{d-1}m_2^r}(\binom{d+r-1}{r}m_1^{d-1}m_2^rH_1^{d-1}H_2^r+\binom{d+r-1}{d}m_1^dm_2^{r-1}H_1^d
				H_2^{r-1})\\&=\frac{1}{\binom{d+r-1}{r}m_1^{d-1}m_2^r}(m_1H_1+m_2H_2)^{d+r-1},
				\end{split}
				\end{equation*}
				and \begin{equation*}
				\begin{split}
				H_1^dH_2^r=\frac{1}{\binom{d+r}{r}m_1^dm_2^r}(m_1H_1+m_2H_2)^{d+r}.
				\end{split}
				\end{equation*}
				
				Hence  $\mu_{\sigma_t}(E)$ is equivalent to the slope function $$\mu_{m_1H_1+m_2H_2}(E)=\frac{(H_1+H_2)^{d+r-1}ch_1(E)}{(H_1+H_2)^{d+r}ch_0(E)}.$$ 
				
			The proof is complete.
		\end{proof}
		\begin{remark}
		In fact, the inequalities in (1) and (3) are the same, we state in this way for the symmetry of the theorem.
		\end{remark}

	If  $S$ is an  abelian variety, the quadratic inequalities in (1) and (2) can be written in a simpler form since we get rid of the Todd classes.

	\begin{corollary}
Let $X$ be a smooth projective variety of dimension $d>0$,  $A$ be an abelian variety of dimension $r>0$.   Let $p,q$ be the projections from $X\times A$ to $X$ and $A$ respectively. Suppose we have $$H_1=c_1(p^*\mathcal{O}_{X}(1)), H_2=c_1(q^*\mathcal{O}_{A}(1)),$$ where $\mathcal{O}_{X}(1),\mathcal{O}_{A}(1)$ are ample line bundles on $X$ and $A$ respectively. 
Then if $E$ is semi-stable sheaf with respect to  $m_1H_1+m_2H_2$, where $m_1,m_2\in\mathbb{Z}_{>0}$ are two positive integers, we have the following inequalities.

	(1) $b_r(E)a_{r-1}(E)-b_{r-1}(E)a_r(E)\geq 0$.

	(2) For any $1\leq i\leq r-1$, if $$b_r(E)a_{r-1}(E)-b_{r-1}(E)a_r(E)=\cdots=b_r(E)a_{i}(E)-b_{i}(E)a_r(E)=0,$$ then $b_r(E)a_{i-1}(E)-b_{i-1}(E)a_r(E)\geq 0$.

	Here $$a_k(E)=-H_1^{d-1}H_2^k\cdot ch_{r+1-k}(E),$$ $$b_k(E)=H_1^{d}H_2^k\cdot ch_{r-k}(E)$$  for any $0\leq k\leq r$.
	\end{corollary}

	\begin{proof}
		This follows directly from Theorem \ref{Bogomolov's inequality}.
	\end{proof}

		We end with the following example, which  illustrates why the leading quadratic inequality can be viewed as a weak version of Bogomolov's inequality.
		
		\begin{example}
			
				Let $X$ be a smooth projective curves over $\mathbb{C}$, $\sigma=(coh(X),-deg+i\cdot rank)$ be the classical slope stability condition on $D(X)$, and $S$ be a smooth projective curve of genus $g$. Suppose $E$ is an object in  the global heart $\mathcal{A}_S$, i.e., a coherent sheaf on  $X\times S$.  
				
			Let us denote the Chern characters of $E$ by $ch(E)=(r, n_1l_1+n_2l_2+\delta, v)$, where $l_1\in H^2(X,\mathbb{Z})\otimes H^0(S,\mathbb{Z})$, $l_2\in H^0(X,\mathbb{Z})\otimes H^2(S,\mathbb{Z})$, $\delta\in H^1(X,\mathbb{Z})\otimes H^1(S,\mathbb{Z})$ and $n_1,n_2\in \mathbb{Z}$. Then the polynomial can be written as $$L_E(n)=-(v+(n+1-g)\cdot ch_1(E)l_2)+i(r\cdot (n+1-g)+ch_1(E)\l_1),.$$ Hence, we have $a_1(E)=-ch_1(E)l_2=n_1$, $a_0(E)=-v+(g-1)n_1$, $b_1(E)=r$, $b_0(E)=n_2+r-rg$. The first quadratic inequality in Theorem 5.1 in this case becomes \begin{equation*}
				\begin{split}
				b_1(E)a_0(E)-a_1(E)b_0(E)&=(-v+(g-1)n_1)\cdot r+n_1\cdot (n_2+r-rg)\\ &=ch_1(E)l_1\cdot ch_1(E)l_2-rv\\ &= \frac{1}{2}(ch_1(E)^2-\delta^2-2ch_0(E)ch_2(E))\geq 0.
				\end{split}
				\end{equation*} This is equivalent to Bogomolov's inequality if $\delta=0$. In general, it is a weak version since $\delta^2\leq 0$ by Hodge index theorem.
			
		\end{example}

\bibliographystyle{alpha}
\bibliography{bibfile}

\end{document}